\documentclass[12pt,reqno]{amsart}
\usepackage{amsmath}
\usepackage{amsfonts}
\usepackage{amssymb}
\usepackage{amsthm}
\usepackage[alphabetic]{amsrefs} 
\usepackage{algorithm,setspace}
\usepackage{algpseudocode}
\usepackage{booktabs}
\usepackage{mathrsfs}
\usepackage{mdframed}
\usepackage{caption}
\usepackage[most]{tcolorbox}
\usepackage[makeroom]{cancel}
\usepackage{array}
\usepackage[shortlabels]{enumitem}
\usepackage[hidelinks]{hyperref}
\hypersetup{
    colorlinks=true,
    linkcolor=magenta,      
    citecolor=magenta,     
    filecolor=magenta,
    urlcolor=magenta,
}
\usepackage[noabbrev,nosort,capitalise]{cleveref}

\usepackage[margin=1in,headsep=30pt]{geometry}
\usepackage{setspace}

\pdfpagewidth=\paperwidth
\usepackage{xcolor}
\usepackage{color,soul}
\usepackage{tikz}
\usepackage{tikz-cd}
\usepackage{tkz-euclide}
\usepackage[latin1]{inputenc}
\usepackage{fancyhdr}
\usepackage{stmaryrd}
\usetikzlibrary{
  graphs,
  graphs.standard
}
\usepackage{titlesec}

\titleformat{\section}
  {\normalfont\Large\bfseries}{\thesection }{1em}{}
\titleformat{\subsection}
  {\normalfont\large\bfseries}{\thesubsection.}{1em}{}
\titleformat{\subsubsection}
  {\normalfont\normalsize\itshape}{\thesubsubsection.}{1em}{}
\newcolumntype{C}{>{$}c<{$}}

\newmdenv[
  leftline=true,
  rightline=false,
  topline=false,
  bottomline=false,
  linewidth=2pt,
  linecolor=gray!50,
  skipabove=\baselineskip,
  skipbelow=0pt,
  innerleftmargin=10pt,
  innerrightmargin=5pt,
  innertopmargin=5pt,
  innerbottommargin=5pt
]{leftbar}

\newcommand{\arxiv}[1]
{\bgroup\color{blue!50}%
\href{http://arxiv.org/abs/#1}{\texttt{\scriptsize arXiv:#1}}\egroup}

\theoremstyle{plain}
\newtheorem{thm}{Theorem}

\newtheorem{conj}[thm]{Conjecture}
\newtheorem{prop}[thm]{Proposition}

\newtheorem{obs}[thm]{Observation}

\theoremstyle{definition}
\newtheorem*{defn}{Definition}
\newtheorem{eg}{Example}

\numberwithin{equation}{section}

\fancypagestyle{Euler}{%
  \fancyhf{} 
  
  \fancyhead[LO]{\sl Groups of Order $n$}
  \fancyhead[RE]{Shihan Kanungo}
  \fancyhead[LE,RO]{\thepage}
}%
\newtheorem{theorem}{Theorem}[section]
\newtheorem{lemma}{Lemma}

\allowdisplaybreaks

\begin{document}

\newpage
\thispagestyle{empty}

\newpage 
\setlength{\parskip}{5pt}
\thispagestyle{empty}
\setcounter{page}{1}

\begin{center}
    {\LARGE \bf Classifying Groups of Certain Orders} 
    \vskip 5pt
    {\normalsize Which integers $n$ have exactly one group of order $n$?}

    \vspace{0.3in}
    {\large  Shihan Kanungo} 
    
    \vspace{0.05in}
    {\small Euler Circle, Palo Alto, CA 94306}
\end{center}
\date{\today}

\vspace{0.2in}
\begin{center}
    {\bf Abstract}

\medskip
\begin{minipage}{12cm}
    {\footnotesize We will first discuss the question of which integers $n$ have exactly one group of order $n$, namely the cyclic group $\mathbb{Z}/n\mathbb{Z}$. We will see that these are the integers that are relatively prime to the Euler totient function $\phi(n)$. Then we discuss how many groups there are of order $p^3$ for each prime $p$. We end with a couple of interesting results and conjectures pertaining to groups of squarefree order.}
\end{minipage}
\end{center}

\setcounter{section}{0}

\vspace{5mm}
\section{Introduction}

One of the first things we learn in abstract algebra is the notion of a cyclic group. For every positive integer $n$, we have the cyclic group $\mathbb{Z}/n\mathbb{Z}$, the group of integers modulo $n$. When $n$ is prime, a simple application of Lagrange's theorem yields that this is the only group of order $n$. We may ask ourselves: what other positive integers have this property? That is, for which positive integers $n$ is every group of order $n$ cyclic?

This is not a new problem; the first solution is attributed to Burnside and has appeared in numerous papers. Dickson \cite{4} determined in 1905 those $n$ for which all groups of order $n$ are abelian. The earliest proof focusing specifically on $n$ for which all groups of order $n$ are cyclic (not just abelian) was given by Szele \cite{7} in 1947.

\begin{defn}
For $n \in \mathbb{N}$ let $f(n)$ denote the number of (isomorphism classes of) groups of order $n$.
\end{defn}

{\sc Question.} Is there a good characterization of $n$ such that $f(n) = 1$?

Considering $f(n)$ for small values of $n$ we see that $f(2)=f(3)=1$ because $2, 3$ are primes. However, for $n=4$ already we have $f(4)>1$.

\begin{eg}
$\mathbb{Z}/4\mathbb{Z}$ and $\mathbb{Z}/2\mathbb{Z}\times\mathbb{Z}/2\mathbb{Z}$ are non-isomorphic: they have different maximal orders for their elements: 4, and 2 respectively. In general, $f(p^2)>1$ because $\mathbb{Z}/p\mathbb{Z}\times\mathbb{Z}/p\mathbb{Z}$ has $p^2$ elements, has no element of order $p^2$, and is therefore not cyclic.
\end{eg}

More generally, if $p^2 \mid n$ for some prime $p$, and if $m=n/p^2$ then $f(n)>1$ because we have at least two non-isomorphic groups of order $n$: \[\mathbb{Z}/p^2\mathbb{Z}\times\mathbb{Z}/m\mathbb{Z} \qquad \text{and} \qquad \mathbb{Z}/p\mathbb{Z}\times\mathbb{Z}/p\mathbb{Z}\times\mathbb{Z}/m\mathbb{Z}.\] Thus, at the very least, we need $n$ to be squarefree.

\begin{defn}
Integers that are divisible by no perfect square other than 1 are called \textit{squarefree integers}.
\end{defn}

For example, $10=2\cdot5$ is squarefree, but $18=2\cdot3^2$ is not.

\medskip
\begin{obs}
If $f(n)=1$ then $n$ must be squarefree.
\end{obs}

However, the converse is false, since $f(6)=2$.

\begin{eg}[The dihedral group $D_3$]
$D_3$ has 6 elements and is generated by the two elements $\{\rho, \tau\}$ where $\rho$ is counterclockwise rotation by $2\pi/3$, and $\tau$ is the reflection across the line $y=0$. Moreover, we have the relation $\rho\tau=\tau\rho^{-1}$, so $D_3$ is not abelian.
\end{eg}

In general, for any integer $n>1$, the dihedral group $D_n$ has $2n$ elements, and is generated by $\{\rho, \tau\}$ with $\rho\tau=\tau\rho^{-1}$. Thus $D_n$ is not abelian, and $f(2n)>1$. Thus, we get

\begin{obs}
If $f(n)=1$, then either $n=2$ or $n$ must be a squarefree odd integer.
\end{obs}

Once again, the converse is false, since $f(21)=2$. This time, however, the reason is not so obvious. It is not straightforward to come up with a group of order 21 that is not isomorphic to $\mathbb{Z}/21\mathbb{Z}$. We need to introduce the notion of a semidirect product of groups.

\section{The Semidirect Product}
We start with a definition:
\begin{defn}
Let $H$ and $K$ be groups and let $\psi:K\rightarrow \operatorname{Aut}(H)$ be a homomorphism. Let $G=\{(h,k):h\in H \text{ and } k\in K\}$. Define multiplication on $G$ by
\[ (h_1,k_1)(h_2,k_2)=(h_1\psi(k_1)(h_2),k_1k_2). \]
This multiplication makes $G$ a group of order $|G|=|K||H|$ where the identity of $G$ is $(e_H,e_K)$, and $(h,k)^{-1}=(\psi(k^{-1})(h^{-1}),k^{-1})$ is the inverse of $(h, k)$. The group $G$ is called the \textit{semidirect product} of $H$ and $K$ with respect to $\psi$ (denoted by $H\rtimes_{\psi}K$).
\end{defn}

{\sc Remark} (The semidirect product is the direct product if the homomorphism is trivial). 
Suppose $H$ and $K$ are groups and $\psi:K\rightarrow \operatorname{Aut}(H)$ is the trivial homomorphism, i.e. $\psi(k)=\operatorname{id}$ for all $k\in K$. Then
\[ (h_1,k_1)(h_2,k_2)=(h_1\cdot\psi(k_1)(h_2),k_1k_2)=(h_1h_2,k_1k_2). \]
Hence $H\rtimes_{\psi}K\cong H\times K$.

The set of automorphisms of a group plays a central role in the study of semidirect products. The semidirect product is distinct from the direct product only if there is some non-trivial homomorphism $\psi:K\rightarrow \operatorname{Aut}(H)$, which only happens if $|K|$ divides $|\operatorname{Aut}(H)|$. Let us remind ourselves how to compute $|\operatorname{Aut}(H)|$ when $H$ is a cyclic group.

\begin{prop}
$\operatorname{Aut}(\mathbb{Z}/n\mathbb{Z})\cong(\mathbb{Z}/n\mathbb{Z})^{\times}$ and this is an abelian group of order $\phi(n)$.
\end{prop}
\begin{proof}
Define $\phi : \mathrm{Aut}(\mathbb Z/n(\mathbb Z)\to (\mathbb Z/n\mathbb Z)^\times$ by $\phi(f)=f(1)$. 

We will show that $\phi$ is an isomorphism, i.e., an homomorphism that is injective and surjective. 
First, $\phi$ is a homomorphism because  \[ \phi(fg)=(f\circ g)(1)=f(g(1)) = g(1)f(1).\] To see that $\phi$ is surjective, suppose $k\in (\mathbb Z/n\mathbb Z)^\times$. We need $f\in \mathrm{Aut}((\mathbb Z/n(\mathbb Z)$ such that $\phi(f)=f(1)=k$. To this end, define \[ f, g : \mathbb Z/n\mathbb Z \quad \text{by}\quad f(1)=k, \ g(1)=k^{-1}\] It is easily checked that $g$ is the inverse of $f$. Thus $f\in \mathrm{Aut}((\mathbb Z/n(\mathbb Z)$ and  $\phi(f)=f(1)=k$. Finally, we note that $\phi$ is injective: if $f(1)=1$ for some $f\in \mathrm{Aut}((\mathbb Z/n(\mathbb Z)$, then $f(a)=a$ for all $a\in \mathbb Z/n\mathbb Z$. This shows that $\operatorname{Aut}(\mathbb{Z}/n\mathbb{Z})\cong(\mathbb{Z}/n\mathbb{Z})^{\times}$. It is clear that $(\mathbb Z/n\mathbb Z)^\times$ is an abelian group of order $\phi(n)$. This completes the proof.
\end{proof}

\medskip
The way we will be using semidirect products is by showing that for some values of $n$, there is a group of order $n$ that is isomorphic to a semidirect product of some two of its subgroups. The following proposition makes this idea concrete.

\begin{prop}
Suppose $H$ and $K$ are subgroups of group $G$ such that $H\triangleleft G$, $H\cap K=\{e\}$, and $|G|=|H|\cdot|K|$. Then $G\cong H\rtimes_{\psi}K$ for some $\psi:K\rightarrow \operatorname{Aut}(H)$.
\end{prop}
\begin{proof}
    See \S 5.5, Theorem 12, in \cite{DummitFoote}.
\end{proof}
\noindent \textsc{note:} When this happens, $K$ is said to be a \textit{complement of} $H$ in $G$.

We can now understand why $f(21)=2$. This is because $21=3\cdot7$ and $3 \mid (7-1)$.

\begin{eg}[Groups of order $pq$ with $p\equiv1 \pmod q$]\phantom{.}\newline
Let $G$ be a group of order $pq$ where $p>q$ are distinct primes. Then there is only one Sylow $p$-subgroup $P$, which is therefore normal. Let $Q$ be a Sylow $q$-subgroup. Then $Q$ is a complement of $P$ in $G$, so $G$ is a semidirect product $P\rtimes_{\psi}Q$, for some homomorphism \[\psi:Q\rightarrow \operatorname{Aut}(P).\] Since $q \mid (p-1)$ then $\operatorname{Aut}(P)\cong \operatorname{Aut}(\mathbb{Z}/p\mathbb{Z})$ has a unique subgroup of order $q$, and $\psi$ can be an isomorphism from $\mathbb{Z}/q\mathbb{Z}$ to this subgroup. We can choose a generator for $\mathbb{Z}/q\mathbb{Z}$ to map to a specified element of order $q$ in $\operatorname{Aut}(\mathbb{Z}/p\mathbb{Z})$. So there is, up to isomorphism, a unique semidirect product which is not a direct product. In other words, the number of groups of order $pq$ (up to isomorphism) is 2 if $q \mid (p-1)$ and 1 otherwise.
\end{eg}

We can gather all these observations into one useful lemma.

\begin{lemma}
If $n=pq$ where $p \mid (q-1)$ then there exists a semidirect product of the cyclic group of order $p$ and the cyclic group of order $q$. In particular, $f(n)=2$.
\end{lemma}
\begin{proof}
See Example above.
\end{proof}

It is easy to see that this can, in fact, be extended to any squarefree integer $n$.

\begin{prop}
If $n=p_1 p_2 \cdots p_k$ where the $p_i$ are distinct primes and $p_i \mid (p_j-1)$ for some $i\ne j$, then $f(n)>1$.
\end{prop}
\begin{proof}
If $n=pqm$ where $p \mid (q-1)$ and $\operatorname{gcd}(pq,m)=1$ then there is a (nontrivial) semidirect product $\mathbb{Z}/p\mathbb{Z}\rtimes_{\psi}\mathbb{Z}/q\mathbb{Z}$ and therefore the following two groups \[(\mathbb{Z}/p\mathbb{Z}\rtimes_{\psi}\mathbb{Z}/q\mathbb{Z})\times\mathbb{Z}/m\mathbb{Z} \quad \text{and} \quad \mathbb{Z}/n\mathbb{Z}\] are non-isomorphic groups of order $n$.
\end{proof}

This concludes our discussion about semidirect products. We will be using it in crucial ways throughout the rest of the discussion. We are now ready to state the answer to our question.

\section{The Main Theorem}

At the end of Section 1 we concluded that it was enough to restrict our attention to the odd, squarefree integers. In Section 2 we discovered that if $n$ is squarefree and $p \mid (q-1)$ for some primes $p, q$ dividing $n$, then we have a semidirect product in addition to the cyclic group of order $n$. This is the same as saying that if $f(n)=1$ then $\operatorname{gcd}(n,\phi(n))=1$ where $\phi(n)$ is the Euler totient function.

\begin{lemma}
Let $n$ be an integer. Then the following statements are equivalent:
\begin{enumerate}[itemsep=5pt]
    \item[(a)] $n=p_1 p_2 \cdots p_k$ where the $p_i$ are distinct primes and $p_i \nmid (p_j-1)$ for $i\ne j$
    \item[(b)] $\operatorname{gcd}(n,\phi(n))=1$, where $\phi$ is the Euler $\phi$-function.
\end{enumerate}
\end{lemma}
\begin{proof}
If $n=p_1^{a_1}p_2^{a_2}\cdots p_k^{a_k}$ then $\phi(n)=p_1^{a_1-1}p_2^{a_2-1}\cdots p_k^{a_k-1}(p_1-1)(p_2-1)\cdots(p_k-1)$.
\end{proof}

Lemma 2 allows us to restate Proposition 5 as

\begin{prop}
If $f(n)=1$ then $\operatorname{gcd}(n,\phi(n))=1$.
\end{prop}

And this time the converse also holds! We thus obtain a tidy classification for integers $n$ such that there is exactly one group of order $n$.

\begin{theorem}
For a positive integer $n$, the only group of order $n$ is the cyclic group $\mathbb{Z}/n\mathbb{Z}$ if and only if $\operatorname{gcd}(n,\phi(n))=1$, where $\phi$ denotes the Euler-$\phi$ function.
\end{theorem}

Proposition 6 proves that the condition is necessary. It will take us quite a bit more work to prove that it is also sufficient. We will do so by first discovering that groups of squarefree order satisfying the conditions of the theorem possess a couple of ``nice'' properties, and then showing inductively that those properties force the group to be cyclic.

\section{Groups of Squarefree Order}

We will build our group inductively, out of its subgroups. But what kind of subgroups should we look for? We have seen that abelian groups are one class of groups that we largely understand; in fact we have a precise classification of all finite abelian groups. Therefore we will try to decompose our group into abelian subgroups. Groups that permit such a subdivision are the solvable groups, so our first step is to show that any group of odd squarefree order is solvable.

\begin{prop}
Let $G$ be a group of order $p_1 p_2 \cdots p_k$, where $p_1, p_2, \dots, p_k$ are distinct primes. Then $G$ is solvable.
\end{prop}
\begin{proof}
We proceed by induction on $k$, the number of prime factors.

If $k=1$, then $|G|=p_1$ is prime, so $G$ is cyclic and hence abelian. Therefore $G$ is solvable.

Now assume $k\ge 2$ and that every group whose order is a product of fewer than $k$ distinct primes is solvable. Let $|G|=p_1p_2\cdots p_k$,
where the $p_i$ are distinct, and let $p_k$ be the largest prime dividing $|G|$. Consider the number $n_{p_k}$ of Sylow $p_k$-subgroups of $G$. By Sylow's theorems,
\[
n_{p_k}\equiv 1 \pmod{p_k} \quad \text{and} \quad 
n_{p_k}\mid p_1p_2\cdots p_{k-1}.
\]

Since $n_{p_k}$ divides the product of the primes $p_1,\dots,p_{k-1}$, every prime divisor of $n_{p_k}$ is strictly smaller than $p_k$. Thus $n_{p_k}<p_k$. 
But $n_{p_k}\equiv 1 \pmod{p_k}$, and the only positive integer less than $p_k$ that is congruent to $1$ modulo $p_k$ is $1$. Hence $n_{p_k}=1$. 
Therefore the Sylow $p_k$-subgroup $P$ is unique and hence normal in $G$.

Since $|P|=p_k$, the subgroup $P$ is cyclic of prime order and therefore abelian. In particular, $P$ is solvable. Moreover,
\[
|G/P|=p_1p_2\cdots p_{k-1},
\]
which is a product of $k-1$ distinct primes. By the induction hypothesis, the quotient group $G/P$ is solvable.
We now have a short exact sequence
\[
1 \longrightarrow P \longrightarrow G \longrightarrow G/P \longrightarrow 1,
\]
with both $P$ and $G/P$ solvable. Since an extension of a solvable group by a solvable group is solvable, it follows that $G$ is solvable.

By induction, every group of order $p_1p_2\cdots p_k$, where the $p_i$ are distinct primes, is solvable.
\end{proof}


Proposition 7 guarantees that we have enough abelian subgroups inside $G$. Now we have to find a way to take two of them of the right size and ``glue'' them together. The way we imagine groups being built out of smaller pieces is that if $G$ is a finite group and $H\lhd G$ then $G$ is built out of $H$ and $G/H$. Thus, we can break down a group into smaller pieces if it has a nontrivial normal subgroup of the right index.

\begin{prop}
Let $G$ be a group of order $p_1 p_2 \cdots p_k$, where $p_1, p_2, \dots, p_k$ are distinct primes. Then $G$ has a normal subgroup of prime index.
\end{prop}
Recall that the \textit{commutator subgroup} $[G,G]$ of $G$, is the subgroup generated by all elements of the form $[g,h]=ghg^{-1}h^{-1}$. The quotient $G/[G,G]$ is always abelian. We will be using both these facts in the proof below.
\begin{proof}
By Proposition 7, $G$ is solvable, so the commutator subgroup $G' = [G, G]$ is not equal to $G$. Then $G'$ is either $\{e\}$ or a proper subgroup of $G$. If $G' = \{e\}$, then $G$ is abelian.

Suppose $G'$ is a proper subgroup of $G$. Then (after reordering, if necessary) $|G'| = p_1 p_2 \cdots p_j$ where $1 \le j < k$. So the quotient group $G/G'$ is an abelian group of order $p_{j+1} \cdots p_k$. Therefore by Cauchy's theorem $G/G'$ has a normal subgroup $H/G'$ of order $p_{j+1} \cdots p_{k-1}$. Hence $H$ is also a normal subgroup of $G$ and $|H| = p_1 p_2 \cdots p_{k-1}$, so $[G:H] = p_k$.
\end{proof}

Now we have everything we need to prove Theorem 3.1.

\section{Proof of Theorem 3.1.}

Suppose $\operatorname{gcd}(n,\phi(n))=1$. Then $n=p_1 p_2 \cdots p_k$ for distinct primes $p_i$ and $p_i \nmid (p_j-1)$ for $i\ne j$. We show that $\mathbb{Z}/n\mathbb{Z}$ is the only group of order $n$. We use induction on $k$, the number of prime factors of $n$.

If $k=1$ then $n=p_1$ i.e. $n$ is a prime. Since every group of prime order is cyclic, $\mathbb{Z}/n\mathbb{Z}$ is the only group of order $n$.

Assume that the result is true for $k=r$ i.e. for $n=p_1 p_2 \cdots p_r$, $\mathbb{Z}/n\mathbb{Z}$ is the only group of order $n$. We will show that the result is true for $k=r+1$ i.e. for $n=p_1 p_2 \cdots p_{r+1}$. 

Let $G$ be a group of order $n=p_1 p_2 \cdots p_{r+1}$.
By Proposition 8, $G$ has a normal subgroup, $H$, of index $p_i$ for some $i \in \{1,\dots,r+1\}$. After reordering, if necessary, we can assume that $H=\mathbb{Z}/m\mathbb{Z}$, where $m=p_1 p_2 \cdots p_r$ and $\operatorname{gcd}(m,\phi(m))=1$. Take $K=\mathbb{Z}/p_{r+1}\mathbb{Z}$ and consider the semidirect product of $H$ and $K$. It exists, because the semidirect product of any two groups exists. Since $H$ is a cyclic group, we have \[|\operatorname{Aut}(H)|=(p_1-1)\cdots(p_r-1).\] Then any homomorphism from $K$ to $\operatorname{Aut}(H)$ must be a trivial homomorphism, because otherwise it would contradict $\operatorname{gcd}(n,\phi(n))=1$. Therefore the semidirect product of $H$ and $K$ is actually just the direct product of $H$ and $K$, which is a cyclic group of order $p_1 p_2 \cdots p_{r+1}$. Since $H$ is the only group of order $m$, $\mathbb{Z}/n\mathbb{Z}$ is the only group of order $n=p_1 p_2 \cdots p_{r+1}$.

This establishes Theorem 3.1.

\medskip
Those $n$, for which $f(n)=1$, are tabulated as sequence \href{http://oeis.org/A003277}{A003277} at \href{http://oeis.org}{oeis.org}

\medskip
\begin{center}
\begin{minipage}{10cm}
\begin{spacing}{1.5}
    1, 2, 3, 5, 7, 11, 13, 15, 17, 19, 23, 29, 31, 33, 35, 37, 41, 43, 47, 51, 53, 59, 61, 65, 67, 69, 71, 73, 77, 79, 83, 85, 87, 89, 91, 95, 97, 101, 103, 107, 109, 113, 115, 119, 123, 127, 131, 133, 137, 139, 141, 143, 145, 149, 151, 157, 159, 161, 163, 167, 173 \ldots
\end{spacing}
\end{minipage}

    \vskip -3mm
    
    {\small \color{gray}  From ``The On-Line Encyclopedia of Integer Sequences''}
\end{center}

\medskip
This concludes our discussion of integers $n$ such that there is exactly one group (up to isomorphism) of order $n$. 

\medskip 
We now move on to explore how many groups there are of order $p^3$ for a given prime $p$.

\section{Groups of Order $p^3$}

For each prime $p$, we want to describe the groups of order $p^3$ up to isomorphism. This was done for $p=2$ by Cayley in 1854 and for odd $p$ by Cole \& Glover, H\"older, and Young independently in 1893.

From the cyclic decomposition of finite abelian groups, there are three abelian groups of order $p^3$ up to isomorphism: \[\mathbb{Z}/p^3\mathbb{Z}, \quad \mathbb{Z}/p^2\mathbb{Z}\times\mathbb{Z}/p\mathbb{Z}, \quad \text{and} \quad  \mathbb{Z}/p\mathbb{Z}\times\mathbb{Z}/p\mathbb{Z}\times\mathbb{Z}/p\mathbb{Z}.\] These are nonisomorphic since they have different maximal orders for their elements: $p^3$, $p^2$, and $p$ respectively.

There are two nonabelian groups of order $p^3$ up to isomorphism. The descriptions of these two groups will be different for $p=2$ and $p\ne2$.

\begin{theorem}
A nonabelian group of order 8 is isomorphic to $D_4$ or to $Q_8$.
\end{theorem}

\begin{eg}[The Quaternion Group $Q_8$]\phantom{.}\newline 
$Q_8=\{1,-1,i,j,k,-i,-j,-k\}$ is the group of order 8 with the multiplication rules \[-1=i^2=j^2=k^2=ijk.\] The element 1 represents the identity and $(-1)^2=1$ and $-1$ is in the center (so $i(-1)=-i$, $j(-1)=-j$, etc.). Then $Q_8$ has the following subsets:
\[ \{1\},\{1,-1\},\{1,-1,i,-i\},\{1,-1,j,-j\},\{1,-1,k,-k\},Q_8 \]
Every subgroup of $Q_8$ is normal in $Q_8$. We know that if $G$ is an abelian group then all subgroups of $G$ are normal. However the group $Q_8$ is non-abelian and yet all of its subgroups are normal.
\end{eg}

\begin{theorem}
For primes $p\ne2$, a nonabelian group of order $p^3$ is isomorphic to $\operatorname{Heis}(\mathbb{Z}/p\mathbb{Z})$ or $G_p$, where
\[ \operatorname{Heis}(\mathbb{Z}/p\mathbb{Z})=\left\{\begin{pmatrix}1&a&b\\ 0&1&c\\ 0&0&1\end{pmatrix}:a,b,c\in\mathbb{Z}/p\mathbb{Z}\right\} \]
and
\[ G_p=\left\{\begin{pmatrix}a&b\\ 0&1\end{pmatrix}:a,b\in\mathbb{Z}/p^2\mathbb{Z},a\equiv1\pmod p\right\}=\left\{\begin{pmatrix}1+pm&b\\ 0&1\end{pmatrix}:m,b\in\mathbb{Z}/p^2\mathbb{Z}\right\} \]
\end{theorem}

Keith Conrad \cite{2} summarizes what is known about the count of groups of small $p$-power order.
\begin{itemize}[itemsep=5pt]
    \item There is one group of order $p$ up to isomorphism $(\mathbb{Z}/p\mathbb{Z})$.
    \item There are two groups of order $p^2$ up to isomorphism: $\mathbb{Z}/p^2\mathbb{Z}$ and $\mathbb{Z}/p\mathbb{Z}\times\mathbb{Z}/p\mathbb{Z}$.
    \item There are five groups of order $p^3$ up to isomorphism, but the explicit description of them is not uniform in $p$ since the case $p=2$ needs a separate treatment.
\end{itemize}
For groups of order $p^4$, the count is no longer uniform in $p$: there are 14 groups of order $2^4$ and 15 groups of order $p^4$ for $p\ne2$. This is due to H\"older and Young.

\section{Further Results and Conjectures}

One of the first mathematicians to make advances in the enumeration of finite groups was Otto H\"older. In 1893, he described groups of order $p^3$ and $p^4$. Shortly thereafter, he derived a remarkable formula for the number of groups of order $n$ when $n$ is square-free.

\begin{theorem}[H\"older, 1895]
The number of groups of order $n$, where $n$ is square-free is given by
\[ f(n)=\sum_{m \mid n}\prod_{p}\frac{p^{c(p)}-1}{p-1} \]
where $p$ runs over all prime divisors of $n/m$ and $c(p)$ is the number of prime divisors $q$ of $m$ that satisfy $q\equiv1\pmod p$.
\end{theorem}

A natural question that arises from H\"older's formula is: for $n$ squarefree, can we relate $f(n)$ to $n$ more explicitly? McIver and Neumann determined that $f(n)\le n^4$ for $n$ square-free. An even better bound is known: $f(n)\le\phi(n)$ where $\phi$ is Euler's $\phi$-function. For squarefree $n=p_1 p_2 \cdots p_r$ and greater than 1, this last result implies that
\[ f(n)\le\phi(n)=(p_1-1)(p_2-1)\cdots(p_r-1)<n. \]
Furthermore, if $n$ is even and squarefree, then $p_1=2$ and
\[ f(n)\le\phi(n)=1(p_2-1)\cdots(p_r-1)<\frac{n}{2}. \]

Another direction is to understand the asymptotic behavior of $f(n)$ when $n$ is square-free. In this light, define
\[ M:=\limsup_{n\rightarrow\infty}\frac{\log f(n)}{\log n} \]
where the limit superior ranges just over squarefree integers $n$. Erd\"os, Murty, and Murty have shown that $M=1$. Their proof uses Dirichlet's Theorem on primes in arithmetic progressions, among other techniques.

We mention one final, curious conjecture in the enumeration of finite groups:

\begin{conj}
    The group enumeration function is surjective. 
\end{conj}

\medskip 
That is, for every positive integer $m$, the conjecture asserts that there exists $n$ such that $f(n)=m$. 
This conjecture may well be resolved through consideration of squarefree $n$, largely because of H\"older's formula. Indeed, it has been verified that every $m$ less than 10,000,000 is equal to $f(n)$ for some squarefree $n$.


\end{document}